\documentclass[a4paper, twoside,12pt]{article}
\usepackage{fancyhdr}
\usepackage{fnpos}
 \usepackage[english]{babel}        
\usepackage[T1]{fontenc}          
\usepackage{graphicx}             
\usepackage{makeidx}
\usepackage{fancybox}
\usepackage{framed}
\usepackage{fancyhdr}
\usepackage[margin=1in]{geometry}
\usepackage{graphicx}
\usepackage{titlesec}
\usepackage{amsmath}
\usepackage{amsfonts}   
\usepackage{amssymb}    
\usepackage{amsthm}
\usepackage{mathtools}
\usepackage{verbatim}   
\usepackage{color}
\usepackage{listings}

 {\markboth{B. Amri and M.  Bedhiafi }{\sl Nonsymmetric Opdam's hypergeometric function of type $A_2$}

\setlength{\topmargin}{-0.3in}
\setlength{\topskip}{0.3in}
\setlength{\textheight}{9.5in}
\setlength{\textwidth}{6in}
\setlength{\oddsidemargin}{0.1in}
\setlength{\evensidemargin}{0.1in}

\newtheorem{thm}{Theorem}[section]

\newtheorem{prop}[thm]{Proposition}

\numberwithin{equation}{section}

\renewcommand{\thefootnote}

\author {B\'echir Amri and  Mounir Bedhiafi }
\date{ }
\begin{document}
\title{ A   formula  for the nonsymmetric Opdam's hypergeometric function of type $A_2$ }
\maketitle
\begin{center}
   Universit\'{e} Tunis El Manar, Facult\'{e} des sciences de Tunis,\\ Laboratoire d'Analyse Math\'{e}matique
       et Applications,\\ LR11ES11, 2092 El Manar I, Tunisie.\\
     \textbf{ e-mail:} bechir.amri@ipeit.rnu.tn,   bedhiafi.mounir@yahoo.fr
\end{center}
\begin{abstract}
The aim of this paper is to give an explicit formula  for the nonsymmetric Heckman-Opdam's hypergeometric function of type $A_2$.
This is obtained by differentiating    the corresponding   symmetric   hypergeometric function.

  \end{abstract}
\section{Introduction  }The theory of the  hypergeometric functions associated
to root systems  started in the 1980s with  Heckman and Opdam via a generalization of the spherical functions on Riemannian symmetric spaces  of noncompact type.   Several important aspects are studied by them in a series of publications \cite{O1,O2,O3,O4,O5}. One of the  impressive developments came in 1995s with the work of  Opdam \cite{O4}, where he introduced   a  remarkable family of   orthogonal polynomials ( the so called Opdam's nonsymmetric
polynomials ) as simultaneous eigenfunctions of  Cherednik operators. It contains in particular the presentation of the non-symmetric  hypergeometric functions where their investigations   become an interesting topics in  the theory of special functions
and in the  harmonic analysis.  In this paper, we focus on  the non-symmetric  hypergeometric function  associated
to root systems  of type $A$, for the purpose in finding   an explicit formula for it, as it is done
in the symmetric case \cite{B1, B2,Saw, Bro}.  The setting is   that    non-symmetric  hypergeometric function can be derived   from  the symmeric ones  via application of a suitable  polynomial of   Cherednik operators ( \cite{O5}, cor. 7.6  ).  This paper  deals with the case where the root system is of type $A_2$, using   Opdam's shifted operators  and  Cherednik operators, so the problem semble to be more robust for others $A_n$.
  \par In order to describe our approach let  us be more specific about  $A$-type hypergeometric function.
We assume that the reader is familiar with root systems and their basic properties. As a general reference, we mention Opdam    \cite{O4,O5}.
\par Let  $(e_1,e_2,...,e_n)$ be the standard
basis of $\mathbb{R}^n$ and   $\langle.,.\rangle$   be  the usual inner product for which this basis
is orthonormal. We denote by $\|\,.\,\|$ its Euclidean norm.  Let $\mathbb{V}$ be the hyperplane  orthogonal to the vector
 $e=e_1+...+e_n$. In  $\mathbb{V}$  we consider  the
root  system of type $A_{n-1}$ $$R=\{ e_i-e_j ,\; 1\leq i\neq j\leq n\}$$
with the subsystem of positives roots
 $$R_+=\{ e_i-e_j ,\; 1\leq i<j\leq n\}.$$
The associated Weyl group $W$ is  isomorphic to symmetric group  $S_n$,   permuting the $n$ coordinates.  We define the positive Weyl chamber
$$C=\{x\in \mathbb{V},\; x_1>x_2>...>x_n\}.$$
Denote $\pi_n$  the orthogonal projection  onto  $\mathbb{V}$, which is given by
$$\pi_n(x)=x-\frac{1}{n}\left(\sum_{j=1}^nx_j\right)e, \quad x\in \mathbb{R}^n.$$
 The cone of dominant weights   is the set
 $$P_+=\sum_{j=1}^{n-1}\mathbb{Z}_+\beta_j,\quad \beta_j=\pi_n(e_1+e_2+...+e_j).$$
 For fixed $k>0$,  the  Dunkl-Cherednik operators $T_\xi $, $\xi\in \mathbb{R}^n$, is defined by
\begin{equation}\label{Cherd}
   T_\xi^k =\partial_\xi +k\sum_{i<j}(\xi_i-\xi_j)\frac{1-s_{i,j}}{1-e^{^{x_j-x_i}}}-\langle \rho_k,\xi\rangle
   \end{equation}
where $\displaystyle{\rho_k=\frac{k}{2}\sum_{j=1}^n (n-2j+1)}e_j$ and  $s_{i,j}$  acts on functions of variables $(x_1,x_2,...,x_n)$ by interchanging the variables $x_i$ and $x_j$. For each $\lambda\in \mathbb{V}_\mathbb{C}$  ( the complexification of $\mathbb{V}$ ) there exists a unique holomorphic W-invariant function $F_k(\lambda, .)$ in a W-invariant tubular neighborhood of $\mathbb{V}$ in $\mathbb{V}_\mathbb{C}$ such that
 $$p(T_{\pi(e_1)}^k, ...T_{\pi(e_n)}^k)F_k(\lambda, .)  = p(\lambda)F_k(\lambda, .); \quad F(\lambda, 0) = 1, $$
 for  all symmetric  polynomial $p\in \mathbb{C}[X_1,..,X_n]$. In particular
 \begin{equation}\label{del}
   \Delta_k F_k(\lambda, .)=  \|\lambda\|^2\,F_k(\lambda, .)
 \end{equation}
  where $\Delta_k=\sum_{i=1}^n\left(T_{\pi(e_i)}\right)^2$ is  the Heckman-Opdam Laplacian. Note that the restriction of $\Delta_k$  to the set of W-invariant functions is the differential operator
$$L_k= \Delta+k\sum_{\alpha\in R^+}\coth \frac{\langle\alpha,x\rangle}{2} \;\partial_\alpha+\langle\rho_k , \rho_k\rangle$$
where $\Delta$ is the ordinary Laplace operator.\\
There exists a unique solution $G_k( \lambda,x)$ of the eigenvalue problem
 \begin{eqnarray}\label{FG}
  T_\xi(k)G_k(\lambda,.)=\langle\lambda,\xi \rangle G_k(\lambda,.),\quad \forall \xi\in \mathbb{R}^N, \qquad G_k(\lambda,0)=1.
\end{eqnarray}
holomorphic for all $\lambda$  and for $x$  in $ \mathbb{V}+ iU$ for a neighbourhood $U\subset\mathbb{V}$
of zero.
The function $G_k$ is the so-called   nonsymmetric Opdam's hypergeometric function. \\
If $x\in \mathbb{V}$ then
\begin{equation}\label{est}
 |G_k(\lambda,x)|\leq\sqrt{ n! }\,e^{\max _{w\in W}\langle Re(\lambda), w(x) \rangle}.
\end{equation}
Moreover,  the Heckman-Opdam hypergeometric function $F_k$  can be written as
 \begin{eqnarray}\label{FG}
F_k(\lambda,x)=\frac{1}{n!}\sum_{w\in W} G_k(\lambda,w.x).
\end{eqnarray}
In other words, for $\lambda$ satisfying $\lambda_i-\lambda_j\neq0; \pm k$  we have
\begin{equation}\label{q}
  G_k(\lambda,x)= D_q F_k(\lambda,x),\qquad x\in \mathbb{V}
\end{equation}
where
 $$D_q=\prod_{1\leq i<j\leq n}\left(1-\frac{k}{\lambda_i-\lambda_j}\right)^{-1}
 \prod_{ w\in W, w\neq id}\left(\frac{T_\xi-\langle w(\lambda),\xi\rangle}{\langle \lambda ,\xi\rangle-\langle w(\lambda),\xi\rangle}\right)$$
and  $\xi$ is any element in $\mathbb{V}$ satisfying  $\langle \lambda ,\xi\rangle-\langle w(\lambda),\xi\rangle\neq0$ for all $w\neq id $.
\\However,  (\ref{q}) is far from being applied to  find an expansion of  $G_k$  when an  explicit formula of $F_k$ is given, so
the polynomial $q$ has   degree $n!-1$. It would therefore be desirable to  get   another polynomial that is of suitable  low degree,
this  will be described  in section 3 when the  root system is of type $A_2$.
\section{An integral formula for Heckman-Opdam  hypergeometric function of type A}
In   \cite{B1}   an  explicit and recursive formula on the dimension $n$ for the A-type    Heckman-Opdam's hypergeometric function is obtained as a consequence of  similar formula for Jack polynomials. We have for  $\lambda\in P_+$  and  $x\in C$
\begin{eqnarray*}
 && F_{k,n}(\lambda+\rho_{k,n} ,x)=\\&&\frac{ c_n(\lambda+\rho_{k,n})}{ c_{n-1}(\overline{\lambda}+\rho_{k,n-1})U_n(\widetilde{\lambda})V_n(x)^{2k-1} }
\int_{x_2}^{x_1}...\int_{x_n}^{x_{n-1}} F_{k,n-1} \Big(\pi_{n-1}(\overline{\lambda})+\rho_{k,n-1}^k,\pi_{n-1}(\nu)\Big)\;\\&&
\qquad\qquad\qquad\qquad\qquad \qquad\qquad \qquad\qquad\qquad e^{|\nu|\left(1+\frac{|\overline{\lambda}|}{n-1} \right)} V_{n-1}( \nu)\;W_k(x, \nu)\;d\nu
\end{eqnarray*}
with the following notations
\begin{eqnarray*}
&& c_n(\lambda)=\prod_{\alpha\in R+}\frac{\Gamma(\langle\lambda, \alpha\rangle)\Gamma(\langle \rho_{k,n}, \alpha\rangle+k)}{\Gamma(\langle\lambda,\alpha\rangle +k )\Gamma(\langle \rho_{k,n}, \alpha\rangle)}, \qquad U_n(\lambda)=\prod_{j=1}^{n-1}\beta(\lambda_j+(n-j)k,k)\\&& \widetilde{\lambda}=(\lambda_1-\lambda_n,\lambda_2-\lambda_n,...,\lambda_{n-1}-\lambda_n,0),
\qquad\overline{\lambda}=(\lambda_1-\lambda_n,\lambda_2-\lambda_n,...,\lambda_{n-1}-\lambda_n),\\&&\\&&|\nu|=\nu_1+\nu_2+...+\nu_{n-1};  \qquad \qquad
V_n(x)=\prod_{1\leq i<j\leq n}(e^{x_i}-e^{x_j}),\\&&    W_k(x,\nu)=\prod_{ 1\leq i\leq n;\;1\leq j\leq n-1}|e^{x_i}-e^{\nu_j}|^{k-1}.
\end{eqnarray*}
 It can be simplified  to
\begin{eqnarray}\label{hy}
\nonumber F_{k,n}(\lambda+\rho_{k,n} ,x)&=&\frac{\Gamma(nk)}{ V_n(x)^{2k-1}\Gamma(k)^n}
\int_{x_2}^{x_1}...\int_{x_n}^{x_{n-1}}F_{k;n-1} \Big(\pi_{n-1}(\overline{\lambda})+\rho_{k,n-1}^k,\pi_{n-1}(\nu)\Big)\;\\&&
\qquad\qquad\qquad\qquad\qquad e^{|\nu|\left(1+\frac{|\overline{\lambda}|}{n-1} \right)} V_{n-1}( \nu)\;W_k(x, \nu)\;d\nu
\end{eqnarray}
 By   analytic continuation, according to Carlson's theorem ( see \cite{T}, p. 186 ), this formula    is still valid for all $\lambda\in \mathbb{V}_\mathbb{C}$.
  Indeed,  for     $x\in C$   define the functions of vaiable $\lambda\in \mathbb{V}_{\mathbb{C}}$
\begin{eqnarray*}
&&H_1(\lambda)=e^{-\langle \lambda,x\rangle}F_{k,n}(\lambda+\rho_{k,n},x)
\\ &&H_2(\lambda)= \frac{\Gamma(nk)}{\Gamma(k)^n V_n(x)^{2k-1}}
\int_{x_2}^{x_1}...\int_{x_n}^{x_{n-1}}e^{-\langle\lambda,x\rangle+|\nu||\overline{\lambda}|/(N-1)}\\&&\qquad\qquad \quad\qquad
F_{k,n-1}^k\Big(\pi_{n-1}(\overline{\lambda})+\rho_{k,n-1} ,\pi_{n-1}(\nu)\Big)  e^{|\nu|} V_{n-1}( \nu)\;W_k(x,e^\nu)\;
d\nu
\end{eqnarray*}
  We have from (\ref{est})
\begin{eqnarray*}
|H_1(\lambda)|\leq \sqrt{ n!}\;e^{\langle\rho_{k,n},x\rangle}
\end{eqnarray*}
for all $\lambda\in \mathbb{V}_\mathbb{C}$, $Re(\lambda)\in C$ and $x\in C$. In order to estimate $H_2$ we note  beforehand the following fact,
 for $\nu=(\nu_1,\nu_2,...,\nu_{n-1})$ and $x\in C$ with $x_{i+1}\leq \nu_i\leq x_i$, if we consider
$\nu^*= (\nu_1,\nu_2,...,\nu_{n-1},\nu_n)\in \mathbb{V}$ then we have  $ \nu^*\preceq x$ where $\preceq $ denotes the partial order on $\mathbb{V} $ associated to the dual  cone
 $\sum_{i=1}^{n-1}\mathbb{R}_+(e_i-e_{i+1})$, which implies  that $\langle Re(\lambda), \nu^*\rangle \leq \langle Re(\lambda), x\rangle$,
 for   $\lambda\in \mathbb{V}_\mathbb{C}$ with  $Re(\lambda)\in C$. Applying this fact and (\ref{est}) it follows that
 \begin{eqnarray*}
  &&\left|e^{-\langle   \lambda,\;x\rangle+|\nu||\overline{\lambda}|/(n-1))}
F_{k,n-1}\Big(\pi_{n-1}(\overline{\lambda})+\rho_{k,n-1},\pi_{n-1}(\nu)\Big)
\right|\\&&\qquad\qquad\qquad\qquad\qquad\qquad\qquad\qquad\leq \sqrt{(n-1)!} \;e^{-\langle Re(\lambda),\, x- \nu^*\rangle+ \langle \rho_{k,n-1},\,\pi_{n-1}(\nu) \rangle}
\\&&\qquad\qquad\qquad\qquad\qquad\qquad\qquad\qquad \leq \sqrt{(n-1)!} \;e^{  \langle \rho_{k,n-1},\,\pi_{n-1}(\nu) \rangle}
\end{eqnarray*}
Hence  $H_1$ and $H_2$ are  Holomorphic functions, bounded  for $Re(\lambda)\in C$
and coincide on $P_+\simeq\mathbb{Z}_+^{n-1}$, the Carleson's Theorem yields  $H_1(\lambda)=H_1(\lambda)$ for all $\lambda\in \mathbb{V}_\mathbb{C} $, $Re(\lambda)\in C$
and thus for all $\lambda\in \mathbb{V}_\mathbb{C} $ by analytic continuation.
 \par Now, using the fact that   $\pi_{n-1}(\overline{\rho_{k,n}})=\rho_{k,n-1}^k$ and $|\overline{\rho_{k,n}}|=kn(n-1)/2$,
we state the following final form of our recursive formula.
\begin{thm} For all $\lambda\in \mathbb{V}_\mathbb{C}$ and $x\in C$.
\begin{eqnarray}\label{hy1}
\nonumber F_{k,n}(\lambda,x)=\qquad\qquad\qquad\qquad\qquad\qquad\qquad\qquad\qquad\qquad
\qquad\qquad\\ \frac{\Gamma(nk)}{\Gamma(k)^n V_n(x)^{2k-1}}
\int_{x_2}^{x_1}...\int_{x_n}^{x_{n-1}}F_{k,n-1} \Big(\pi_{n-1}(\overline{\lambda}),\pi_{n-1}(\nu)\Big)\;
e^{|\nu|\left(1-nk/2+|\overline{\lambda}|/(n-1)\right)}\nonumber \\ V_{n-1}( \nu )\;W_k(x, \nu)\;
d\nu\qquad\qquad\qquad\qquad\quad
\end{eqnarray}
\end{thm}
\par
 In the rank-one case, which corresponds to take $n = 2$ and $\mathbb{V}=\mathbb{R}(e_1-e_2)$, the formula  (\ref{hy1}) becomes
\begin{eqnarray*}
F_{k,2}(\lambda,x)&=& \frac{\Gamma(2k)}{\Gamma(k)^2(2\sinh x_1)^{2k-1}}\int_{-x_1}^{x_1} e^{\nu(1-k+2\lambda_1)}
(e^{x_1}-e^{\nu})^{k-1}(e^{\nu}-e^{-x_1})^{k-1}d\nu
\\&=& \frac{\Gamma(2k)}{2^k\Gamma(k)^2(\sinh x_1)^{2k-1}}\int_{-x_1}^{x_1} e^{2\nu\lambda_1}
(\cosh x_1-\cosh\nu)^{k-1}d\nu\\&=&\varphi^{k-1/2,-1/2}_{2i\lambda_1}(x_1)=\;_2F_1\left(\frac{k}{2}-\lambda_1,
\;\frac{k}{2}+\lambda_1,\;k+\frac{1}{2},\;-\sinh^2x_1\right)
\end{eqnarray*}
where  $\varphi^{k-1/2,-1/2}_{2i\lambda_1}$  is a  Jacobi function see (1.4), (3.4) and (3.5) of \cite{koo1}. We recall here various
facts about the Jacobi function  $\varphi^{k-1/2,-1/2}_{i\eta}$  that we shall need later, we  refer to \cite{koo1,koo2}.
\begin{eqnarray}
&&\varphi^{k-1/2,-1/2}_{i\eta}(2t)=\varphi^{k-1/2,k-1/2}_{2i\eta}(t)\label{fi1}\qquad\qquad\qquad\qquad\qquad\qquad\qquad \qquad\qquad\qquad\\
&&\left(\varphi^{k-1/2,-1/2}_{i\eta}\right)'(t)=\frac{1}{(2k+1)}(\eta^2-k^2)\sinh(t)\varphi^{k+1/2,-1/2}_{i\eta}(t)\label{fi2}\\
&&\left(\varphi^{k-1/2,-1/2}_{i\eta}\right)''(t)+2k\coth(t)\left(\varphi^{k-1/2,-1/2}_{i\eta}\right)'(t)=(\eta^2-k^2)\varphi^{k+1/2,-1/2}_{i\eta}(t)\label{fi3}
\end{eqnarray}
\par In the rank-two case, where  $n = 3$ which is our subject in the next section,  Heckman-Opdam's hypergeometric function has the following integral
representation ( we omit here the dependance on $n$ )
\begin{eqnarray}\label{n3}
\nonumber F_{k}(\lambda,x)=&&\frac{\Gamma(3k)}{\Gamma(k)^3}V(x)^{-2k+1}\int_{x_2}^{x_1}\int_{x_3}^{x_{2}}
e^{ (1-\frac{3}{2}(\lambda_3 +k))(\nu_1+\nu_2)}\;\varphi^{k-\frac{1}{2},-\frac{1}{2}}_{i(\lambda_1-\lambda2)}\left(\frac{\nu_1-\nu_2}{2}\right)
\\&&\qquad\qquad\qquad\qquad\qquad\quad(e^{\nu_1}-e^{\nu_2})W_k(\nu,x)\;d\nu.
\end{eqnarray}
where
$$W_k(\nu,x)=\Big((e^{x_1}-e^{\nu_1})(e^{x_1}-e^{\nu_2})(e^{x_2}-e^{\nu_2})(e^{\nu_1}-e^{x_2})(e^{\nu_1}-e^{x_3})(e^{\nu_2}-e^{x_3}))\Big)^{k-1}$$
and
$$V(x)=(e^{x_1}-e^{x_2})(e^{x_1}-e^{x_3})(e^{x_2}-e^{x_3}) .$$
In order to find an expression  for    $F_{k}$ of Laplace type, we write
\begin{eqnarray*}
F_{k}(\lambda,x)&=&\frac{\Gamma(3k)}{4\Gamma(k)^3\left(\prod_{1\leq i<j\leq3}\sinh\left(\frac{x_i-x_j}{2}\right)\right)^{2k-1}}\int_{x_2}^{x_1}\int_{x_3}^{x_{2}}
e^{-3(\nu_1+\nu_2)\lambda_3/2 }
\\&&\varphi^{k-\frac{1}{2},-\frac{1}{2}}_{i(\lambda_1-\lambda2)}\left(\frac{\nu_1-\nu_2}{2}\right)\sinh\left(\frac{\nu_1-\nu_2}{2}\right)
\left(\prod_{1\leq i\leq2;\;1\leq j\leq 3}\sinh\left(\frac{|\nu_i-x_j|}{2}\right)\right)^{k-1}
d\nu.
\end{eqnarray*}
With  the
 change of variables  $$y_1=\frac{\nu_1+\nu_2}{2},\quad t=\frac{\nu_1-\nu_2}{2},$$ we have that
\begin{eqnarray*}
F_{k}(\lambda,x)&=&\frac{\Gamma(3k)}{2^{3k-2}\Gamma(k)^3\left(\prod_{1\leq i<j\leq3}\sinh(\frac{x_i-x_j}{2})\right)^{2k-1}}
\int_{\mathbb{R}^2}
e^{-3\lambda_3y_1 }\varphi^{k-\frac{1}{2},-\frac{1}{2}}_{i(\lambda_1-\lambda_2)}(t)\sinh(t)
\\&&\Big((\cosh(x_1-y_1)-\cosh t)(\cosh t - \cosh(x_2-y_1))(\cosh(x_3-y_1)-\cosh t)\Big)^{k-1}\\&&
 \; \chi_{[x_2,x_1]}(y_1+t)\;\chi_{[x_3,x_2]}(y_1-t) dy_1dt.
\end{eqnarray*}
Now inserting
\begin{eqnarray}\label{2}\nonumber
\varphi^{k-\frac{1}{2},-\frac{1}{2}}_{i(\lambda_1-\lambda_2)}(t)&=&\frac{\Gamma(2k)}{2^k\Gamma(k)^2(\sinh t)^{2k-1}}\int_{\mathbb{R}} e^{y_2(\lambda_1-\lambda_2)}
(\cosh (t)-\cosh y_2)^{k-1}\chi_{[-1,1]}\left(\frac{y_2}{t}\right)dy_2,
\end{eqnarray}
with the use of    Fubini's Theorem  and the fact that
\begin{eqnarray*}
\chi_{[-1,1]} (\frac{y_2}{t}) \chi_{[x_2,x_1]}(y_1+t)\;\chi_{[x_3,x_2]}(y_1-t)=
\chi_{\max(|y_2|,|y_1-x_2|)\leq t\leq \min(y_1-x_3, x_1-y_1)},
\end{eqnarray*}
if follows that
\begin{eqnarray}\label{14}
F_{k}(\lambda,x)= \int_{\mathbb{R} }\int_{\mathbb{R} }
 e^{3(\lambda_1+\lambda_2)y_1+ (\lambda_1-\lambda_2)y_2} R_k(x,y_1,y_2)dy_1dy_2
\end{eqnarray}
where
\begin{eqnarray*}
&&R_k(x,y_1,y_2)=\frac{\Gamma(2k)\Gamma(3k)}{2^{4k-2}\Gamma(k)^{5}\left(\prod_{1\leq i<j\leq3}\sinh\left(\frac{x_i-x_j}{2}\right)\right)^{2k-1}}
\\&&\qquad\int_{\max(|y_2|,|y_1-x_2|)}^{\min(y_1-x_3,x_1-y_1)}\left( \frac{\cosh t-\cosh y_1}{\sinh ^2t}\right)^{k-1}
\Big(\prod_{i=1}^3|\cosh(x_i-y_1)-\cosh(t)|\Big)^{k-1}dt,
\end{eqnarray*}
if $\max(|y_2|,|y_1-x_2|)\leq \min(y_1-x_3, x_1-y_1)$ and $R_k(x,y_1,y_2)=0$, otherwise.
 We should note here that
 the condition  $$\max(|y_2|,|y_1-x_2|)\leq \min(y_1-x_3, x_1-y_1)$$ is equivalent to
$$x_3\leq y_1\pm y_2,-2y_1\leq x_1$$
an thus  equivalents to $y=( y_1+ y_2, y_1- y_2,-2y_1)\in co(x)$, the convex hull of the orbit $W.x$, see proposition 3.6 of \cite{B1}. Also, we have
$y=y_1 \sqrt{6}\,\varepsilon_1+y_2\sqrt{2}\,\varepsilon_2$ in the orthonormal basis of $\mathbb{V}$
  $$\varepsilon_1=(e_1+e_2-2e_3)/\sqrt{6},\quad \varepsilon_2=(e_1-e_2)/\sqrt{2}.$$
Making  the change of variables $ z_1= \sqrt{6}\,y_1,\; z_2=\sqrt{2}\, y_2$ in  the formula (\ref{14}) and identify $\mathbb{R}^2$ with $\mathbb{V}$ via the basis $(\varepsilon_1,\varepsilon_2)$ we  finally write
\begin{eqnarray}\label{15}
F_{k}(\lambda,x)= \int_{co(x)}
 e^{ \langle\lambda,z\rangle} \mathcal{N}_k(x,z)dz.
\end{eqnarray}
where $\mathcal{N}_k(x,z)= R_k\left(x, z_1/\sqrt{6},z_2/\sqrt{2}\right)/\sqrt{12}$.
\par Next, for   $f\in C^\infty(\mathbb{V})$, $W$- invariant,  we define $V_k(f)$ the $W$- invariant function
 on $\mathbb{V}=\mathbb{R}\varepsilon_1\oplus \mathbb{R}\varepsilon_2$ by
\begin{equation*}
  V_k(f)(x)=\int_{co(x)}
 f(z)  \mathcal{N}_k(x,z)dz; \qquad x\in C.
\end{equation*}
  \begin{prop}
 For $f\in C_c^\infty(\mathbb{V})$ and $x\in C$, we have the  following intertwining property
$$  \quad \Delta_k(V_k(f))= V_k(\Delta(f)).$$
\end{prop}
\begin{proof}
   By inversion formula for  Fourier transform and Fubini's Theorem
\begin{eqnarray}\label{VV}
 V_k(f)(x)=\int_{ \mathbb{R}^2} \int_{co(x)}
 \widehat{f}(\xi) e^{i\langle \xi,z\rangle}   \mathcal{N}_k(x,z)dzd\xi
= \int_{ \mathbb{R}^2}
 \widehat{f}(\xi) F_{k}(i\xi,x)d\xi.
 \end{eqnarray}
 Here  we define the  Fourier transform of $f$ by
 $$ \widehat{f}(\xi)=\frac{1}{2\pi}\int_{ \mathbb{R}^2}
 f(t) e^{-i\langle \xi,t\rangle}dt$$
From a general estimates of Heckman-Opdam's hypergeometric function ( see for example corollary 6.2 in \cite{O5}), the  last integral of (\ref{VV})  is a $C^\infty$ as a function of $x$ and  then by (\ref{del}) one has
\begin{eqnarray*}
 \Delta_k(V_k(f))(x) =- \int_{\mathbb{R}^2}\|\xi\|^2
 \widehat{f}(\xi) F_{k}(i\xi,x)d\xi=\int_{\mathbb{R}^2}\widehat{\Delta f}(\xi) F_{k}(i\xi,x)d\xi=V_k(\Delta(f)(x),
\end{eqnarray*}
which proves the desired fact.
  \end{proof}
 \section{ Nonsymmetric Opdam's  hypergeometric function of   type $A_2$}
We  begin with some backgrounds from  Heckman-Opdam theory of hypergeometric  functions  associated to a root system $R$ with Weyl group $W$ of a finite-dimensional  vector space $ \mathfrak{a}$ , we refer   to  \cite{O4,O5}  for a more detailed treatment. For a regular weight $\lambda\in P_+$
 (the set of dominant weights)
    we denote  by  $P_k (\lambda,.)$ the   Hekman-Opdam  Jacobi polynomial and by $E_k(\lambda,.)$ the non symmetric  Opdam  polynomial. In the following  we collect  some   properties  and  relationships.
\begin{itemize}
\item [(i)] $\displaystyle{P_k(\lambda,x)= \sum_{w\in W}E_k(\lambda,wx)}$
  \item [(ii)] $\displaystyle{\sum_{w\in W} det(w) E_k(\lambda+\delta,wx)=V(x) P_{k+1}(\lambda ,x)}$, where $\displaystyle{\delta=\frac{1}{2}\sum_{\alpha\in R^+} \alpha}$ and
      \\ $\displaystyle{V(x) =\prod_{\alpha\in R^+} (e^{\alpha/2}-e^{-\alpha/2}) }$
\item [(iii)] $\displaystyle{P_k(\lambda,0)= \frac{c_k(\rho_k)}{c_k(\lambda+\rho_k)}}$, where
$\displaystyle{\rho_k(\lambda)=\frac{1}{2}\sum_{\alpha\in R+}k_\alpha}\alpha$,
$\quad\displaystyle{c_k(\lambda)=\prod_{\alpha\in R+}\frac{\Gamma(\langle\lambda,\check{\alpha} \rangle)}{\Gamma(\langle\lambda,\check{\alpha}\rangle +k_{\alpha})}}$
and $\check{\alpha}=2\alpha/|\alpha|^2$.
 \item [(iv)] $P_k(\lambda ,x)= P_k(\lambda ,0)F_k(\lambda+\rho_k,x)$
\item [(v)]$\displaystyle{E_k(\lambda,x)= \frac{P_k(\lambda, 0)}{|W|}\;G_k(\lambda+\rho_k,x)}$.
\end{itemize}
 From these facts we give  the following consequence.
\begin{prop}\label{k+1}
For all $\lambda, x\in \mathfrak{a}$,
\begin{equation}\label{shift}
 \sum_{w\in W} det(w) G_k (\lambda,w.x)=
d_k(\lambda )V(x) F_{k+1}(\lambda,x),
\end{equation}
  where,
$$d_k(\lambda)=\; |W|\; \frac{c_{k+1}(\rho_{k+1} )}{c_k(\rho_k)}\;\frac{c_k(\lambda)}{c_{k+1}(\lambda  )}.$$
\end{prop}
\begin{proof}
This is closely related to $(ii)$, where   by using the above relations one can write   for regular $\lambda\in P_+$,
$$  \sum_{w\in W} det(w) G_k(\lambda+\rho_{k+1},x)=d_k(\lambda+\rho_{k+1})V(x) F_{k+1}(\lambda+\rho_{k+1},x),$$
since we have $\rho_{k}+\delta=\rho_{k+1}$.
This identity  can be extended  for all $\lambda\in \mathfrak{a}$ via analytic continuation by means of Carlson's theorem.

\end{proof}
\par
Let us return now to the space $\mathbb{V}$, for $n=3$ and investigate (\ref{shift}). In this case
we have
$$d_k(\lambda)=\frac{1}{(2k+1)(3k+1)(3k+2)}\;(\lambda_1-\lambda_2+k)(\lambda_2-\lambda_3+k)(\lambda_1-\lambda_3+k).$$
We   introduce the antisymmetric function
\begin{equation}\label{F*}
 F^*_k(\lambda,x)=\frac{1}{6}\sum_{w\in S_3} det(w) G_k(\lambda,x)=\frac{d_k(\lambda)}{6}V(x) F_{k+1}(\lambda,x)
\end{equation}
and   the following notations:
 \begin{eqnarray*}
I_k(\lambda,\nu)&=&\varphi^{k-\frac{1}{2},-\frac{1}{2}}_{i(\lambda_1-\lambda2)}\left(\frac{\nu_1-\nu_2}{2}\right)
\;e^{ -\frac{3}{2}(\lambda_3 +k)(\nu_1+\nu_2)};\qquad
\gamma_k=\frac{\Gamma(3k)}{\Gamma(k)^3}\\
L_k(\lambda,\nu)&=&\frac{\left( \varphi^{k-\frac{1}{2},-\frac{1}{2}}_{i(\lambda_1-\lambda2)}\right)'\left(\frac{\nu_1-\nu_2}{2}\right)
}{\lambda_1-\lambda_2-k}\;\;e^{ -\frac{3}{2}(\lambda_3 +k)(\nu_1+\nu_2)}.
\end{eqnarray*}
So we can write
\begin{eqnarray}\label{F_k}
F_k(\lambda,x)=\gamma_k V(x)^{-2k+1}\int_{x_2}^{x_1}\int_{x_3}^{x_{2}}
I_k(\lambda,\nu) e^{\nu_1+\nu_2}(e^{\nu_1}-e^{\nu_2})W_k(\nu,x)d\nu.
\end{eqnarray}
The integral representation  for the functions $F_k^*$ is given in   the following.
\begin{prop}
We have for $\lambda\in \mathbb{V_\mathbb{C}}$ and $x\in C$,
\begin{eqnarray}\label{inf*}
  F_{k}^*(\lambda,x)= \frac{\gamma_{k}}{4k^2}V(x)^{-2k } \int_{x_2}^{x_1}\int_{x_3}^{x_{2}}
L_k(\lambda,\nu)p(x,\nu)W_{k}(\nu,x)d\nu.
\end{eqnarray}
 where
 \begin{eqnarray*}
 p(x,\nu)&=&-2be^{2(\nu_1+\nu_2)}+(ab+3)e^{\nu_1+\nu_2}(e^{\nu_1}+e^{\nu_2})-
 2a(e^{2\nu_1}+e^{2\nu_2})\\&&-2(b^2+a)e^{\nu_1+\nu_2}+4b(e^{\nu_1}+e^{\nu_2})-6
 \end{eqnarray*}
and with $a=e^{x_1}+e^{x_2}+e^{x_3}$ and $b=e^{-x_1}+e^{-x_2}+e^{-x_3}$.
\end{prop}
\begin{proof}
We first  write
\begin{eqnarray*}
&&(\lambda_1-\lambda_2+k)(\lambda_1-\lambda_3+k)(\lambda_2-\lambda_3+k)\\&&
 =\frac{(\lambda_1-\lambda_2+k)}{4}\Big((-3\lambda_3+k)^2+2k
(-3\lambda_3+k)-((\lambda_1-\lambda_2)^2-k^2)\Big).
\end{eqnarray*}
The formula  (\ref{n3}) for parameter $k+1$, together with  (\ref{fi2}) yield
\begin{eqnarray*}
  (\lambda_1-\lambda_2+k)F_{k+1}(\lambda,x) = 2(2k+1)\gamma_{k+1}   V(x)^{-2k-1} \int_{x_2}^{x_1}\int_{x_3}^{x_{2}}
 L_k(\lambda,\nu) W_{k+1}(\nu,x)d\nu.
\end{eqnarray*}
So using integration by parts,
\begin{eqnarray*}
&&( -3\lambda_3+k)(\lambda_1-\lambda_2+k)F_{k+1}(\lambda,x) \\&&= 2(2k+1)\gamma_{k+1}   V(x)^{-2k-1}\int_{x_2}^{x_1}\int_{x_3}^{x_{2}}
L_k(\lambda,\nu) (4k-(\partial_{\nu_1}+\partial_{\nu_2}))W_{k+1}(\nu,x)d\nu.
\end{eqnarray*}
and
\begin{eqnarray*}
&&( -3\lambda_3+k)^2(\lambda_1-\lambda_2+k)F_{k+1}(\lambda,x) \\&&= 2(2k+1)\gamma_{k+1}   V(x)^{-2k-1}\int_{x_2}^{x_1}\int_{x_3}^{x_{2}}
L_k(\lambda,\nu) (\partial_{\nu_1}+\partial_{\nu_2}-4k)^2W_{k+1}(\nu,x)d\nu.
\end{eqnarray*}
In the same way with the use of (\ref{fi3}),
\begin{eqnarray*}
&&((\lambda_1-\lambda_2)^2-k^2)(\lambda_1-\lambda_2+k)F_{k+1}(\lambda,x) \\&&= -2(2k+1)\gamma_{k+1}V(x)^{-2k-1} \int_{x_2}^{x_1}\int_{x_3}^{x_{2}}
e^{  -\frac{3}{2}(\lambda_3+k)(\nu_1+\nu_2)}
(\lambda_1-\lambda_2+k) \\&&\qquad\qquad\qquad\qquad\qquad\qquad \varphi^{k-\frac{1}{2},-\frac{1}{2}}_{i(\lambda_1-\lambda2)} \left(\frac{\nu_1-\nu_2}{2}\right)
(\partial_{\nu_1}-\partial_{\nu_2}  )W_{k+1}(\nu,x)d\nu\\&&=
2(2k+1)\gamma_{k+1}V(x)^{-2k-1}\int_{x_2}^{x_1}\int_{x_3}^{x_{2}}
e^{  -\frac{3}{2}(\lambda_3+k)(\nu_1+\nu_2)}\frac{\left( \varphi^{k-\frac{1}{2},-\frac{1}{2}}_{i(\lambda_1-\lambda2)}\right)'\left(\frac{\nu_1-\nu_2}{2}\right)
}{\lambda_1-\lambda_2-k}\\&&\qquad\qquad\qquad\qquad\qquad\qquad\qquad(\partial_{\nu_1}-\partial_{\nu_2}  )^2W_{k+1}(\nu,x)d\nu\\&&
-4k(2k+1)\gamma_{k+1}V(x)^{-2k-1}\int_{x_2}^{x_1}\int_{x_3}^{x_{2}}
e^{  -\frac{3}{2}(\lambda_3+k)(\nu_1+\nu_2)}
\frac{\left( \varphi^{k-\frac{1}{2},-\frac{1}{2}}_{i(\lambda_1-\lambda2)}\right)'\left(\frac{\nu_1-\nu_2}{2}\right)
}{\lambda_1-\lambda_2-k}\\&&\qquad\qquad\qquad\qquad\qquad\qquad\qquad \left(\frac{e^{\nu_1}+e^{\nu_2}}{e^{\nu_1}-e^{\nu_2}}\right)(\partial_{\nu_1}-\partial_{\nu_2}  ) W_{k+1}(\nu,x)d\nu
\\&&=2(2k+1)\gamma_{k+1}V(x)^{-2k-1}\int_{x_2}^{x_1}\int_{x_3}^{x_{2}}
 L_k(\lambda,\nu) \Big\{(\partial_{\nu_1}-\partial_{\nu_2}  )^2\\&&\qquad\qquad\qquad\qquad\qquad\qquad\qquad-2k \left(\frac{e^{\nu_1}+e^{\nu_2}}{e^{\nu_1}-e^{\nu_2}}\right)(\partial_{\nu_1}-\partial_{\nu_2}  )\Big\}W_{k+1}(\nu,x)d\nu.
\end{eqnarray*}
Hence it follows that
 \begin{eqnarray*}
&& F_{k}^*(\lambda,x)\nonumber  \\&&= \frac{\gamma_{k}}{4k^2}V(x)^{-2k } \int_{x_2}^{x_1}\int_{x_3}^{x_{2}}
L_k(\lambda,\nu)
\Big\{24k^2+  2k\left(\frac{e^{\nu_1}+e^{\nu_2}}{e^{\nu_1}-e^{\nu_2}}\right)(\partial_{\nu_1}-\partial_{\nu_2})\nonumber\\&&\qquad\qquad\qquad\qquad\qquad -10k(\partial_{\nu_1}+\partial_{\nu_2})+4\partial_{\nu_1}\partial_{\nu_2}
\Big\}
W_{k+1}(\nu,x)d\nu.
\end{eqnarray*}
 Now, it is straightforward computation to verify that
 $$p(\nu,x)=\Big\{24k^2+  2k\left(\frac{e^{\nu_1}+e^{\nu_2}}{e^{\nu_1}-e^{\nu_2}}\right)(\partial_{\nu_1}-\partial_{\nu_2}) -10k(\partial_{\nu_1}+\partial_{\nu_2})+4\partial_{\nu_1}\partial_{\nu_2}
\Big\}
W_{k+1}(\nu,x)$$
\end{proof}
\par Next, we can provide an integral expansion  for   $G_k$ by   differentiating
(\ref{F_k}) and (\ref{inf*}). Let us  introduce the operators
\begin{eqnarray*}
D_k&=&D_k(\lambda)= (\lambda_1-\lambda_3+2k)T_{\pi(e_1)}+(\lambda_2-\lambda_3+k)T_{\pi(e_2)}+\tau(\lambda)+k(\lambda_1-\lambda_3)+k^2\\
D_k^*&=&D_k^*(\lambda)=(\lambda_1-\lambda_3-2k)T_{\pi(e_1)}+(\lambda_2-\lambda_3-k)T_{\pi(e_2)}+\tau(\lambda)-k(\lambda_1-\lambda_3)+k^2
\end{eqnarray*}
where $\tau(\lambda)=\lambda_1^2+\lambda_2^2+\lambda_1\lambda_2$.
 Recall that  $\pi(e_1)=(\frac{2}{3},-\frac{1}{3},-\frac{1}{3})$ and $\pi(e_2)=(-\frac{1}{3},\frac{2}{3},-\frac{1}{3})$.
Our  main result from this section is the following.
\begin{thm}
For $\lambda,\,x\in \mathbb{V}$ we have
\begin{equation}\label{T}
 (\tau(\lambda)-k^2)G_k(\lambda,x)= D_k(F_k(\lambda,x))+D_k^*(F^*_k(\lambda,x)).
\end{equation}
\end{thm}
\begin{proof}
The proof is purely computational. Using  the following fact, see (5.1 ) of  \cite{O4},
   $$wT_\xi w^{-1}=T_{w.\xi}-\sum_{\alpha\in R_+,\;w.\alpha\;\in R_-}k\,\langle\alpha,\xi\rangle s_{w\alpha}.$$
 we  obtain
\begin{eqnarray*}
T_{\pi(e_1)}\;(G_k(\lambda,x))&=&\lambda_1G_k(\lambda,x),\\
T_{\pi(e_1)}\;(G_k(\lambda,s_{1,2}.x))&=&\lambda_2G_k(\lambda,s_{1,2}.x)-kG_k(\lambda,x),\\
T_{\pi(e_1)}\;(G_k(\lambda,s_{2,3}.x))&=&\lambda_1G_k(\lambda,s_{2,3}.x),\\
T_{\pi(e_1)}\;(G_k(\lambda,s_{1,3}.x))&=&\lambda_3G_k(\lambda,s_{1,3}.x)-k G_k(\lambda,\sigma.x)-kG_k(\lambda, x),\\
T_{\pi(e_1)}\;(G_k(\lambda,\sigma.x))&=&\lambda_2G_k(\lambda,\sigma.x))-k G_k(\lambda,s_{2,3}.x)\\
T_{\pi(e_1)}\;(G_k(\lambda,\sigma^2.x))&=&\lambda_3G_k(\lambda,\sigma^2.x))-kG_k(\lambda,s_{1,2}.x)-kG_k(\lambda,s_{2,3}.x),
\end{eqnarray*}
and
\begin{eqnarray*}
 T_{\pi(e_2)}\;(G_k(\lambda,x))&=&\lambda_2G_k(\lambda,x),\\
T_{\pi(e_2)}\;(G_k(\lambda,s_{1,2}.x))&=&\lambda_1G_k(\lambda,s_{1,2}.x)+kG_k(\lambda,x),\\
T_{\pi(e_2)}\;(G_k(\lambda,s_{2,3}.x))&=&\lambda_3G_k(\lambda,s_{2,3}.x)-kG_k(\lambda,x)\\
T_{\pi(e_2)}\;(G_k(\lambda,s_{1,3}.x))&=&\lambda_2G_k(\lambda,s_{1,3}.x)-k G_k(\lambda,\sigma^2.x)+kG_k(\lambda,\sigma. x),\\
T_{\pi(e_2)}\;(G_k(\lambda,\sigma.x))&=&\lambda_3G_k(\lambda,\sigma.x))-k G_k(\lambda,s_{1,2}.x),\\
T_{\pi(e_2)}\;(G_k(\lambda,\sigma^2.x))&=&\lambda_1G_k(\lambda,\sigma^2.x))+kG_k(\lambda,s_{2,3}.x).
\end{eqnarray*}
where  $\sigma=s_{1,3}s_{1,2}$. Thus we have
\begin{eqnarray*}
T_{\pi(e_1)}\;(F_k(\lambda,x))=&& \frac{1}{6} \Big\{(\lambda_1-2k)G_k(\lambda,x)+ (\lambda_2-k)G_k(\lambda,s_{1,2}.x)\\&& +
 (\lambda_1-2k)G_k(\lambda,s_{2,3}.x)+
 \lambda_3G_k(\lambda,s_{1,3}.x)\\&&+(\lambda_2-k) G_k(\lambda,\sigma.x)   +\lambda_3G_k(\lambda,\sigma^2.x)    \Big\}
\end{eqnarray*}
\begin{eqnarray*}
T_{\pi(e_2)}\;(F_k(\lambda,x))=&& \frac{1}{6}\Big\{ \lambda_2G_k(\lambda,x)+ (\lambda_1-k)G_k(\lambda,s_{1,2}.x)\\&& +
 (\lambda_3+k)G_k(\lambda,s_{2,3}.x)+
 \lambda_2G_k(\lambda,s_{1,3}.x)\\&&+(\lambda_3+k) G_k(\lambda,\sigma.x)   +(\lambda_1-k)G_k(\lambda,\sigma^2.x)     \Big\},
\end{eqnarray*}
\begin{eqnarray*}
T_{\pi(e_1)}\;(F_k^*(\lambda,x))=&& \frac{1}{6}\Big\{(\lambda_1+2k)G_k(\lambda,x)- ( \lambda_2+k)G_k(\lambda,s_{1,2}.x)\\&&
 -(\lambda_1+2k)G_k(\lambda,s_{2,3}.x)-
 \lambda_3G_k(\lambda,s_{1,3}.x)\\&&+(\lambda_2+k) G_k(\lambda,\sigma.x)   +\lambda_3G_k(\lambda,\sigma^2.x)    \Big\},
\end{eqnarray*}
and
\begin{eqnarray*}
T_{\pi(e_2)}\;(F_k^*(\lambda,x))=&& \frac{1}{6} \Big\{\lambda_2G_k(\lambda,x)- (\lambda_1+k)G_k(\lambda,s_{1,2}.x)\\&& +
 (-\lambda_3+k)G_k(\lambda,s_{2,3}.x)
 -\lambda_2G_k(\lambda,s_{1,3}.x)\\&&+(\lambda_3-k) G_k(\lambda,\sigma.x)   +(\lambda_1+k)G_k(\lambda,\sigma^2.x)\Big\}.
\end{eqnarray*}
So, formula (\ref{T}) can be checked by   a straightforward calculations.
\end{proof}

 \end{document}